\documentclass[11pt,reqno]{amsart}
\usepackage{fullpage,a4wide}

\usepackage{graphicx}
\usepackage[draft]{hyperref}
\usepackage{amsmath,amsopn,amssymb,amsfonts,stmaryrd}
\usepackage{verbatim}
\usepackage{amsthm}
\usepackage{mathtools}
\usepackage{color}
\usepackage{enumitem}
\usepackage[framemethod=TikZ]{mdframed}
\usepackage{bbm}
\usepackage{mathrsfs}
\usepackage{booktabs}
\usepackage{caption}
\usepackage{bm}
\usepackage{tensor}
\usepackage{cleveref}

\renewcommand{\H}{\mathbb{H}}

\newcommand{\CL}{\mathcal{L}}
\newcommand{\CM}{\mathcal{M}}
\newcommand{\CN}{\mathcal{N}}

\newcommand{\sgn}{\mbox{sgn}}

\newcommand{\SU}{\mathrm{SU}}

\newcommand{\N}{\mathbb N}
\newcommand{\C}{\mathbb C}

\theoremstyle{plain}
\newtheorem{theorem}{Theorem}[section]
\newtheorem{corollary}[theorem]{Corollary}
\newtheorem{lemma}[theorem]{Lemma}
\newtheorem{proposition}[theorem]{Proposition}

\newtheorem*{remark*}{Remark}
\newtheorem*{remarks*}{Remarks}

\theoremstyle{definition}

\newtheorem*{notation*}{Notation}
\newtheorem*{claim*}{Claim}
\numberwithin{equation}{section}

\renewcommand{\sgn}{\textnormal{sgn}}

\def\a{\alpha}
\def\b{\beta}

\def\p{\rho}

\def\w{\omega}

\def\z{\zeta}
\def\th{\theta}

\def\s{\sigma}
\def\g{\gamma}

\def\a{\alpha}
\def\b{\beta}

\def\p{\varrho}

\def\w{\omega}

\def\z{\zeta}
\def\th{\theta}

\def\s{\sigma}
\def\g{\gamma}

\def\GG{\Gamma}

\def\LL{\Lambda}

\def\GG{\Gamma}

\def\LL{\Lambda}

\newcommand{\re}{{\rm Re}}

\renewcommand{\sgn}{{\rm sgn}}
\newcommand{\R}{\mathbb R}

\renewcommand{\binom}[2]{\left(\begin{smallmatrix}#1\\\\#2\end{smallmatrix}\right)}

\setlist[itemize]{noitemsep, topsep=0pt}

\makeatletter
\newcommand{\vast}{\bBigg@{3}}
\newcommand{\Vast}{\bBigg@{5}}
\makeatother

\renewcommand{\pmod}[1]{\ \left( \mathrm{mod} \, #1 \right)}

\allowdisplaybreaks

\makeatletter
\@namedef{subjclassname@2020}{%
	\textup{2020} Mathematics Subject Classification}
\makeatother
\begin{document}
\title{Asymptotics of commuting $\ell$-tuples in symmetric groups and log-concavity}
\author{Kathrin Bringmann, Johann Franke, and Bernhard Heim}
\address{Department of Mathematics and Computer Science\\Division of Mathematics\\University of Cologne\\ Weyertal 86-90 \\ 50931 Cologne \\Germany}
\email{kbringma@math.uni-koeln.de}
\email{jfranke@uni-koeln.de}
\email{bheim@uni-koeln.de}
\begin{abstract}
Denote by $N_{\ell}(n)$ the number of $\ell$-tuples of elements in the symmetric group $S_n$ with commuting components, normalized by the order of $S_n$. In this paper, we prove asymptotic formulas for $N_\ell(n)$. In addition, general criteria for log-concavity are shown, which can be applied to $N_\ell(n)$ among other examples. Moreover, we obtain a Bessenrodt–Ono type theorem which gives an inequality of the form $c(a)c(b) > c(a+b)$ for certain families of sequences $c(n)$.

\end{abstract}
\maketitle
\section{Introduction and statement of results}
In this paper we consider asymptotics for commuting $\ell$-tuples in $S_n$, where $S_n$ denotes the symmetric group for $n,\ell \in \N$. To be more precise, let $|M|$ be the cardinality of a set $M$ and define
\begin{equation}
C_{\ell,n}:=
\left\{ (\pi_1, \ldots, \pi_{\ell}) \in S_n^{\ell} :  \pi_j  \pi_k = \pi_k  
\pi_j \text{ for } 1 \leq j,k \leq \ell \right\}.
\end{equation}
The numbers $|C_{\ell,n}|$ are divisible by $\left| S_n \right|$ and appear as the specialization of the $\ell$-th orbifold characteristic to the $n$-th symmetric product of a manifold of ordinary Euler characteristic $1$ (see \cite{ABDV23} for a combinatorial approach and Theorem 2.1 of \cite{BF98}).
In this paper, we prove asymptotics and log-concavity of 
	\begin{equation} N_{\ell}(n):= \frac{\left| C_{\ell,n}\right|}
	{\left| S_n \right|}.
	\end{equation}
\par
Bryan and Fulman \cite{BF98}\footnote{The result of Bryan--Fulman had been strongly influenced by Stanley (\cite{BF98}, acknowledgments).} proved the following.\footnote{Further proofs have been obtained by \cite{ABDV23} and \cite{Wh13}.
	The work of Bryan and Fulman can be considered as
	a generalization of a combinatorial formulae due to Macdonald \cite{Ma62}
	and Hirzebruch and H\"ofer \cite{HH90}. }
\begin{theorem}[Bryan and Fulman, see also \cite{ABDV23}, p. 3]\label{th:BF}
	For $\ell \in \N$, we have
	\begin{equation*} 
		\sum_{n=0}^{\infty} \vert{C}_{\ell,n}\vert  \frac{q^n}{n!} =\prod_{n=1}^{\infty} \left( 1-q^n\right)^{- g_{\ell-1}(n)} =
		\exp \left( \sum_{n=1}^{\infty} g_{\ell}(n)  \frac{q^n}{n}\right),
	\end{equation*}
	where $g_{\ell}(n)$ denote the number of subgroups of $\mathbb{Z}^{\ell}$ of index $n$ (we refer to the work of Lubotzky and Segal \cite{LUSE} for more background on the group theoretic interpretation). 
\end{theorem}

Since $g_2(n)=\sigma_1(n)$, where for $m \in \N$, $\sigma_m(n) := \sum_{d \vert n} d^m$, we obtain that
$N_{2}(n)=p(n)$, where $p(n)$ denotes the number of partitions of $n$ (for more background, see \cite{An98, On04}).
Thus,
$N_{2}(n)$ equals to the number of conjugacy classes in $S_n$
\cite{ET68}. Note that $\ell =3$ reveals an interesting connection to topology 
as highligted by Britnell (see introduction of \cite{Br13}) due to work of
Liskovets and Medynkh \cite{LM09}. Namely $N_{3}(n)$ counts the number
of non-equivalent $n$-sheeted coverings of a torus.

Recently, Neuhauser and one of the  authors \cite{HN23} proved that
$N_{\ell}(n)$ is log-concave for $n\geq 20$ for almost all $\ell$ if and only if $ n \equiv 0 \pmod{3}$. Moreover it was shown by Nicolas (\cite{Ni78}, Proposition 1) and reproved by DeSalvo and Pak (\cite{DP15}, Theorem 1.1) that
$N_2(n)$ is log-concave for $n \geq 26$. In \cite{HN23} it was conjectured that
$N_3(n)$ is also log-concave for $n \geq 22$ (numerically verified there  for $n \leq 10^5$).
Based on numerical experiments they also 
speculated that for fixed $\ell$, $N_{\ell}(n)$ is log-concave for almost all $n$.

In this paper, using results by Bridges, Brindle, and two of the authors in \cite{BBBF23}, we prove asymptotic formulas for $N_\ell(n)$ for arbitrary $\ell$, and partially answer the question (5) posed in Section 6 there. Note that we are in the situation of multiple poles which is not covered by the classical result of Meinardus \cite{Meinardus}. 
\begin{theorem} \label{thm:Main-Nl}  For $\ell \geq 6$ we have,\footnote{Here and in the following such series are meant as asymptotic expansions.} as $n \to \infty$,
	\begin{align*} 
	N_\ell(n) \sim \frac{(\ell-1)!^{\frac{1}{2\ell}} \sqrt{Z_\ell}}{\sqrt{2\pi \ell} n^{\frac{\ell+1}{2 \ell}}} \exp\left( \frac{\ell \Gamma(\ell)^{\frac{1}{\ell}}Z_\ell}{\ell-1} n^{\frac{\ell-1}{\ell}} + \sum_{k=2}^{\ell} A_{\ell,k} n^{\frac{\ell-k}{\ell}}\right)\left( 1 + \sum_{j=1}^\infty \frac{B_{\ell,j}}{n^{\frac{j}{\ell}}}\right),
	\end{align*}
	where $Z_\ell := (\zeta(2) \cdot \zeta(3) \cdots \zeta(\ell))^{\frac{1}{\ell}}$ for certain $A_{\ell,k}$ and $B_{\ell,j}$.
\end{theorem}
\begin{remarks*}\ 
	\begin{enumerate}[leftmargin=*,label=\textnormal{(\arabic*)}]
		\item We give a more explicit description for the constants $A_{\ell,k}$ in \eqref{eq:Ak}. Theoretically, the values $B_{\ell,j}$ can also be calculated explicitly. They result from the calculation methods in \cite{BBBF23} by a rather complicated procedure.
		\item The cases $\ell \in \{3,4,5\}$ are simpler and of special interest. They are treated separately in \Cref{thm:N3} and \Cref{thm:N345}.
	\end{enumerate}
\end{remarks*}
It turns out that asymptotics like the one in \Cref{thm:Main-Nl} are sufficient to prove log-concavity\footnote{For the definition, more background and applications in combinatorics see Stanley \cite{St89} and Section $5$ of Br\"and\'en \cite{Braenden}.} of sequences. Recall that a sequence $a_n$ is called {\it log-concave} if $a_n^2\ge a_{n+1}a_{n-1}$. We prove the following general result. 
\begin{theorem}\label{thm:log}
	Assume that $c(n)$ is a sequence with
	\begin{equation*}
	c(n) \sim \frac C{n^\kappa} \exp\left(\sum_{\lambda\in\mathcal S}A_\lambda n^\lambda \right) \sum_{\mu\in\mathcal T} \frac{\beta_\mu}{n^\mu} \qquad (n \to \infty).
	\end{equation*}
Here $\kappa\in\R$, $\mathcal S\subset\mathbb Q^+ \cap (0,1)$ is finite, $\mathcal T\subset\mathbb Q_{0}^+$, $C$, $A_\lambda$, $\beta_\mu\in\R$ with $\beta_0 = 1$. Let $\lambda^{\ast} := \max\{\lambda \in \mathcal{S}: A_{\lambda} \neq 0\}$ and assume that $A_{\lambda^*}>0$. Then, for $n$ sufficiently large, $c(n)$ is log-concave.\footnote{The result can probably also be extended to real exponents, but this case is not required for any of our applications and thus we do not allow this case here.} 
\end{theorem}
In particular, from \Cref{thm:log} and \Cref{thm:Main-Nl}, we conclude the following.
\begin{corollary} \label{cor:log-concave-Nl}
	Let\footnote{Note that the case $\ell=1$ is trivially true.} $\ell \geq 2$. For $n$ sufficiently large $N_{\ell}(n)$ is log-concave.
\end{corollary}
\begin{remark*} In \Cref{cor:pkrr} we give as further examples partitions into $k$-gonal numbers, as well as $n$-dimensional representation numbers of the groups $\mathfrak{su}(3)$ and $\mathfrak{so}(5)$. 
	\end{remark*}
	\par
In 2016, Bessenrodt and Ono (\cite{BessOno}, Theorem 2.1) proved that if $a, b \in \N$ satisfy $a, b > 1$ and $a+b > 8,$ then $p(a)p(b) \geq p(a+b)$ with equality if and only if $\{a,b\} = \{2,7\}$. We also show a Bessenrodt--Ono type theorem for general sequences, which implies Bessenrodt's and Ono's result on $p(n)$ for $a$, $b$ sufficiently large.
\begin{theorem} \label{thm:BO}
		Let $c(n)$ be a sequence satisfying
	\begin{equation*}
	c(n) \sim \frac C{n^\kappa} \exp\left(\sum_{\lambda\in\mathcal S} A_\lambda n^\lambda\right)
	\end{equation*}
	with $\kappa \in \R$, $\mathcal{S}\in\R\cap(0,1)$, $C, A_{\lambda}\in\R$. Let $\lambda^{\ast}:=\max\{\lambda\in\mathcal{S}: A_{\lambda}\neq 0\}$ with $A_{\lambda^*} > 0$.
	If $a,b\gg1$, then
	\begin{equation*}
	c(a)c(b) > c(a+b).
	\end{equation*}
	\end{theorem}
Again, we can apply this to the sequences $N_\ell(n)$.
\begin{corollary}
	Let $\ell \geq 2$. We have, for $a,b\gg1$
	\begin{equation*}
	N_{\ell}(a)N_{\ell}(b) > N_{\ell}(a+b).
	\end{equation*}
\end{corollary}
The paper is organized as follows. In Section 2, we recall known results. In Section 3 we prove \Cref{thm:Main-Nl}. In Section \ref{S:Log}, we show our main results concerning log-concavity and give some examples. In Section \ref{sect:BO} we provide a proof of \Cref{thm:BO}. In Section 6, we state some open questions.
\section*{Acknowledgments}
The first  author received funding from the European Research Council (ERC) under the European Union’s Horizon 2020 research and innovation programme (grant agreement No. 101001179).
\section{Preliminaries}
\subsection{Results from \cite{BBBF23}}
We require some results from \cite{BBBF23}. Let $f:\N\to\N_0$, set $\LL:=\N\setminus f^{-1}(\{0\})$, and for $q=e^{-z}$ ($z\in\C$ with $\re(z) > 0$), define
\begin{align*} 
	G_f(z):=\sum_{n\ge0} p_f(n)q^n := \prod_{n\ge1} \frac{1}{\left(1-q^n\right)^{f(n)}},\qquad L_f(s) := \sum_{n\ge1} \frac{f(n)}{n^s}.
\end{align*}
We let $\mathcal P$ be the set of poles of $L^*_f(s) := \GG(s)\z(s+1)L_f(s)$, and for $R>0$ we denote by $\mathcal P_R$ 
the union of the poles of $L_f^*$ greater than $-R$ with $\{0\}$. We require the following key properties of these objects:
\begin{enumerate}[leftmargin=*,label=\rm{(P\arabic*)}]
	\item\label{paper:main:1} All poles of $L_f$ are real. Let $\a>0$ be the largest pole of $L_f$. There exists $L\in\N$, such that for all primes $p$, we have $|\LL\setminus (p\N\cap\LL)|\ge L>\frac\a2$.
	\item\label{paper:main:3} Condition \ref{paper:main:3} is attached to $R\in\R^+$. The series $L_f(s)$ converges for some $s\in\C$, has a meromorphic continuation to $\{s\in\C:\re(s)\ge-R\}$, and is holomorphic on the line $\{s\in\C:\re(s)=-R\}$. The function $L_f^*(s)=\GG(s)\z(s+1)L_f(s)$ has only real poles $0<\a:=\g_1>\g_2>\dots$ that are all simple, except the possible pole at $s=0$, that may be double. 
	\item\label{paper:main:4} For some $a<\frac\pi2$, in every strip $\s_1\le\s\le\s_2$ in the domain of holomorphicity of $L_f(s)$, we uniformly have, for $s=\s+it$,
	\begin{align*}
		L_f(s) = O_{\sigma_1, \sigma_2}\left(e^{a|t|}\right), \qquad |t| \to \infty.
	\end{align*}
\end{enumerate}
We have the following asymptotic behavior of $p_f(n)$.
\begin{theorem}[\cite{BBBF23}, Theorem 1.4] \label{paper:BBBF:T:main2}
	Assume \ref{paper:main:1} for $L\in\N$, \ref{paper:main:3} for $R>0$, and \ref{paper:main:4}. Then, for some $M,N\in\N$, we have
	\begin{equation*}
		p_f(n) = \frac{C}{n^b}
		\exp\left(A_1n^\frac{\a}{\a+1}+\sum_{j=2}^M A_jn^{\a_j}\right)\left(1+\sum\limits_{j=2}^N \frac{B_j}{n^{\b_j}} + O_{L,R}\left(n^{-\min\left\{\frac{2L-\a}{2(\a+1)},\frac{R}{\a+1}\right\}}\right)\right).
	\end{equation*}
	Here, $0\le\a_M<\a_{M-1}<\cdots\a_2<\a_1=\frac{\a}{\a+1}$ are given by $\CL$ 
	\begin{align*}
		\CL := \frac{1}{\a+1}\mathcal P_R + \sum_{\mu\in\mathcal P_R} \left(\frac{\mu+1}{\a+1}-1\right)\N_0.
	\end{align*}
	The exponents $0<\b_2<\b_3<\dots$ are given by $\CM+\CN$, where $\CM$ and $\CN$ are defined by 
	\begin{align*}
		\CM &:= \left\{ \theta = \frac{\a}{\a+1}n + \sum_{\mu\in\mathcal P_R} \left(1 - \frac{\mu+1}{\a+1}\right)n_\mu  \colon n, n_\mu \in \N_0, \theta \in \left[0,\frac{R+\a}{\a+1}\right) \right\}, \\
		\CN &:= \left\{\sum_{j=1}^K b_j\th_j : b_j,K\in\N_0,\th_j\in(-\CL)\cap\left(0,\frac{R}{\a+1}\right)\right\}.
	\end{align*}
	The coefficients $A_j$ and $B_j$ can be calculated explicitly; $A_1$, $C$, and $b$ are given by 
	\begin{align*}
		A_1 &:= \left(1+\frac1\a\right)(\w_\a\GG(\a+1)\z(\a+1))^\frac{1}{\a+1},\qquad C := \frac{e^{L_f'(0)}(\w_\a\GG(\a+1)\z(\a+1))^\frac{\frac12-L_f(0)}{\a+1}}{\sqrt{2\pi(\a+1)}},\\
		b &:= \frac{1-L_f(0)+\frac\a2}{\a+1},
	\end{align*}
	where $\omega_\nu := \mathrm{Res}_{s=\nu} L_f(s)$. Moreover, if $\a$ is the only positive pole of $L_f$, then we have $M=1$.
\end{theorem}
The situation of exactly two positive poles was worked out explicitly in Theorem 4.4 of \cite{BBBF23}. 
\begin{theorem}\label{paper:BBBF:T:TwoPoleAsymptotics} 
	Assume that $f:\N\to\N_0$ satisfies the conditions of \Cref{paper:BBBF:T:main2} and that $L_f$ has exactly two positive poles $\a>\b$, such that $\frac{\lambda+1}{\lambda}\b<\a\le\frac{\lambda}{\lambda-1}\b$ for some $\lambda\in\N$. Then 
	\begin{multline*}
	p_f(n) = \frac{C}{n^b}\exp\left(A_1n^\frac{\a}{\a+1}+A_2n^\frac{\b}{\a+1}+\sum_{k=3}^{\lambda+1} A_kn^{\frac{(k-1)\b}{\a+1}+\frac{k-2}{\a+1}+2-k}\right)\\
	\times \left(1+\sum_{j=2}^{N} \frac{B_j}{n^{\nu_j}} + O_{L,R}\left(n^{-\min\left\{\frac{2L-\a}{2(\a+1)},\frac{R}{\a+1}\right\}}\right)\right),\qquad (n\to\infty),
	\end{multline*}
	with
	\begin{equation*}
	A_1 := (\w_{\a}\GG(\a+1)\z(\a+1))^\frac{1}{\a+1}\left(1+\frac1\a\right),\qquad A_2 := \frac{\w_{\b}\GG(\b)\z(\b+1)}{(\w_\a\GG(\a+1)\z(\a+1))^\frac{\b}{\a+1}},
	\end{equation*}
	and for $k \geq 3$
	\begin{multline*}
	A_k := K_k + \frac{c_1^\frac{1}{\a+1}}{\a}\sum_{m=1}^\lambda \binom{-\a}{m}\sum_{\substack{0\le j_1,\dots,j_\lambda\le m\\j_1+\ldots+j_\lambda=m\\j_1+2j_2+\ldots+\lambda j_\lambda =k-1}} \binom{m}{j_1,j_2,\dots,j_\lambda}\frac{K_2^{j_1}\cdots K_{\lambda+1}^{j_\lambda}}{c_1^\frac{m}{a+1}}\\
	+ \frac{c_2}{\b c_1^\frac{\b}{a+1}}\sum_{m=1}^\lambda \binom{-\b}{m}\sum_{\substack{0\le j_1,\dots,j_\lambda\le m\\j_1+\ldots+j_\lambda=m\\j_1+2j_2+\ldots+\lambda j_\lambda=k-2}} \binom{m}{j_1,j_2,\dots,j_\lambda}\frac{K_2^{j_1}\cdots K_{\lambda+1}^{j_\lambda}}{c_1^\frac{m}{a+1}}.
	\end{multline*}
	Here, ${m \choose m_1, m_2, \ldots, m_k} := \frac{m!}{m_1! m_2! \cdots m_k!}$ with $\sum_{j=1}^k m_j = m$ denotes the multinomial coefficient. The $\nu_j$ run through $\CM+\CN$, the $K_j$ are given in Lemma 4.3 of \cite{BBBF23}, and $c_1$, $c_2$, and $c_3$ are defined by
		\begin{align*} 
		c_1 := \omega_{\a} \Gamma(\a+1)\z(\a+1), \quad c_2 := \omega_{\b} \Gamma(\b+1) \z(\b+1), \quad c_3 := L_f(0).
		\end{align*}
\end{theorem}
\begin{remark*}
By Lemma 4.3 of \cite{BBBF23}, the first values of $K_j$ are given by
\begin{align*}
	K_1 &= c_1^\frac{1}{\a+1},
	\hspace{3mm} K_2 = \frac{c_2}{(\a+1)c_1^\frac{\b}{\a+1}},
	\hspace{3mm} K_3 = \frac{c_2^2(\a-2\b)}{2(\a+1)^2c_1^\frac{2\b+1}{\a+1}}, \\
	K_4 &= \frac{c_2^3\left(2\a^2-9\a\b-2\a+9\b^2+3\b\right)}{6(\a+1)^3c_1^\frac{3\b+2}{\a+1}}, \\
	K_5 &= \frac{c_2^4(6 \a^3-44 \a^2 \b-15 \a^2+96 \a \b^2+56 \a \b+6 \a-64 \b^3-48 \b^2-8 \b)}{24(\a+1)^4 c_1^{\frac{4\b + 3}{\a + 1}}}.
\end{align*}
\end{remark*}
\begin{remark*}
	As the number of positive poles increases, the situation quickly becomes much more complicated. The focus of this paper is, in particular, the case of three poles.
\end{remark*}
We also require the behavior of a certain saddle point function. We adopt the notation for the coefficients of asymptotic expansions from \cite{BBBF23} and write for a sequence $g(n)$ 
\begin{align*}
g(n) = \sum_{j = 1}^{N} \frac{a_{g,j}}{n^{\nu_{j}}} + O_R\left(n^{-R}\right), \quad (\nu_1 < \nu_2 < \cdots < \nu_R < R).
\end{align*}
\begin{proposition}[Corollary 3.4 of \cite{BBBF23}] \label{prop:BBBF:paper:cor:rhosaddleexp1}
	Let $\Phi_f := \mathrm{Log}(G_f)$ and assume that $f\colon\N\to\N_0$ satisfies the conditions of \Cref{paper:BBBF:T:main2}. Let $\p_n > 0$ solve\footnote{We go back and forth between functions and sequences in our notations here. For more details see Subsection 2.2 in \cite{BBBF23}.}  
	\begin{align*} 
	-\Phi'_f(\varrho) = n.
	\end{align*}
	Then 
	\begin{align*}
	\p_{n} = \sum_{1 \leq j \leq N_\varrho} \frac{a_{\varrho,j}}{n^{\nu_{\varrho,j}}} + O\left(n^{-\frac{R}{\a+1}-1}\right) 
	\end{align*}
	with $a_{\p,1}=a_{-\Phi_f',1}^\frac{1}{\a+1} = (\w_\a\GG(\a+1)\z(\a+1))^{\frac{1}{\a+1}}$ and we have 
	\begin{align*}
	\{\nu_{\p,j}\colon 1\leq j\leq N_\p\} =\left(\frac{1}{\a+1} - \sum_{\substack{\mu\in \textcolor{black}{\mathcal{P}_R}}}\left( \frac{\mu+1}{\a+1} - 1 \right) \N_0\right)\cap \left[\frac{1}{\a + 1},\frac{R}{\a + 1} + 1\right).
	\end{align*}
\end{proposition}
 The following is required for a more explicit investigation of the involved constants.
\begin{lemma}[Lemma 4.1 of \cite{BBBF23}]\label{lem:BBBF:paper:L:mainasy}
	Let $f:\N\to\N_0$ satisfy the conditions of \Cref{paper:BBBF:T:main2}. Then
	\begin{align*}
	p_f(n) = \frac{e^{n\p_{n}}G_f(\p_n)}{\sqrt{2\pi}}\left(\sum\limits_{j=1}^M \frac{d_j }{n^{\nu_j}} + O_{L,R}\left(n^{-\min\left\{\frac{L+1}{\a+1}, \frac{R+\a}{\a+1}+\frac{\a+2}{2(\a+1)}\right\}}\right)\right)
	\end{align*}
	for some $M\in\N$, $d_1=\frac{1}{\sqrt{\a+1}}(\w_\a\GG(\a+1)\z(\a+1))^\frac{1}{2(\a+1)}$, and $\nu_j$ runs through
	\[
	\frac{\a+2}{2(\a+1)} + \frac{\a}{\a+1}\N_0 + \left(-\sum_{\mu\in\mathcal P_R} \left(\frac{\mu+1}{\a+1}-1\right)\N_0\right) \cap \left[0,\frac{R+\a}{\a+1}\right).
	\]
	In particular, we have $\nu_1 =\tfrac{\a+2}{2(\a+1)}$.
\end{lemma}
\subsection{Results from complex analysis} We require the Lagrange inversion formula.
\begin{lemma}[Corollary 11.2 of \cite{Char}] \label{lem:power-series-inverse}
	Let $\phi:B_r(0)\to D$, where for $r \in \R^+$ as usual $B_r(0) := \{ z \in \C \colon |z| < r\}$, be a holomorphic function such that $\phi(0)=0$ and $\phi'(0)\ne0$, with $\phi(z)=:\sum_{n\ge1}a_nz^n$. Then $\phi$ is locally biholomorphic and its local inverse has a power series expansion $\phi^{-1}(w)=:\sum_{k\ge1}b_kw^k$, with
	\begin{equation*}
	b_k = \frac{1}{ka_1^{k}} \sum_{\substack{\ell_1,\ell_2, \ell_3 \dots \geq 0 \\ \ell_1 + 2\ell_2 + \dots = k-1}} (-1)^{\ell_1 + \ell_2 +\ell_3+\dots} \frac{k \cdots (k-1+\ell_1+\ell_2+\dots)}{\ell_1! \ell_2! \ell_3! \cdots}\left( \frac{a_2}{a_1}\right)^{\ell_1} \left( \frac{a_3}{a_1}\right)^{\ell_2} \cdots.
	\end{equation*}
\end{lemma}
To study the asymptotics of $N_{\ell}(n)$ we first have to investigate the analytic properties of the Dirichlet series attached to the sequence $g_{\ell -1}(n)$ from Theorem \ref{th:BF} with $f_{\ell}:=g_{\ell -1}$. Solomon \cite{So79} (see also \cite{LUSE}, Theorem 15.1, p. 296) proved that, for $\ell \geq 2$,
\begin{equation}\label{eq:Lg-Zeta}
L_{f_{\ell}}(s) =
 \prod_{k=0}^{\ell -2} \zeta(s- k).
\end{equation}
For the following investigations the following result is essential; its proof follows from standard properties of the Riemann zeta function.
\begin{proposition} \label{prop:Poles-of-Lg}
	The function $L_{f_\ell}$ has one pole at $s=1$ for $\ell=2$, two poles $s\in \{1,2\}$ for $\ell=3$, and three poles $s\in \{\ell-3, \ell-2,\ell-1\}$ for $\ell \geq 4$.
	If $\ell \geq 4$, then we have for $\nu \in \{ \ell-3, \ell-2, \ell-1 \}$
	\begin{align*}
	\mathrm{Res}_{s = \nu} L_{f_\ell}^*(s) = (\nu-1)!\zeta(\nu+1)  \prod_{\substack{0\le k\le\ell-2\\ k \not= \nu-1}} \zeta(\nu - k)
	\end{align*}
	and $L_{f_\ell}(0) = 0$. Additionally, for $\ell \geq 6$ the function $L^*_{f_\ell}$ only has simple poles in $s \in \{\ell-3, \ell-2, \ell-1\}$. For $\ell \in \{4,5\}$, $L_{f_\ell}^*$ has an additional simple pole in $s=0$, with residue
	\begin{align*}
	\mathrm{Res}_{s=0} L_{f_4}^*(s) =  \frac{\zeta'(-2)}{24}, \qquad \mathrm{Res}_{s=0} L_{f_5}^*(s) = \frac{\zeta '(-2)}{2880}.
	\end{align*}
\end{proposition}
\section{Asymptotic expansions for $N_\ell(n)$}
\subsection{Exponent sets} 
For our calculations, we need the following lemma which follows by a direct calculation.
\begin{lemma} 
	\label{lem:Set-L} For $\ell \geq 3$ and $f = f_\ell$, we have, as $R \to \infty$,
	\begin{align*}
		\mathcal{L} & = \frac{\ell-1}{\ell} - \frac{1}{\ell} \N_0, \quad \mathcal{M} = \frac{1}{\ell} \N_0, \quad \mathcal{N} = \frac{1}{\ell} \N_0.
		\end{align*}
\end{lemma}

\subsection{The case $\ell = 2$}\ This case is the partition function and is classical. In the setting of the present paper it is 
\begin{equation*} 
N_2(n) = p(n) = \frac{e^{\pi\sqrt{\frac{2n}{3}}}}{4\sqrt{3}n}\left(1+\sum_{j=1}^N \frac{B_j}{n^\frac j2}+O_N\left(n^{-\frac{N+1}{2}}\right)\right)
\end{equation*}
for certain $B_j$ and was treated in \cite{BBBF23}.
\subsection{The case $\ell = 3$}{\hspace{0cm}} The following theorem gives the asymptotic behavior of $N_{3}(n)$.
\begin{theorem} \label{thm:N3} 
	We have, as $n \to \infty$,
	\begin{equation*} 
		N_{3}(n) \sim \frac{e^{-\frac{\zeta'(-1)}2  -\frac{\pi ^2}{288 \zeta (3)}}\zeta(3)^{\frac{11}{72}}}{2^{\frac{11}{24}} \cdot 3^{\frac{47}{72}} \cdot \pi^{\frac{11}{72}} \cdot n^{\frac{47}{72}}}  \exp\left(  \frac{(3\pi)^{\frac23} \zeta(3)^{\frac13}}{2} n^{\frac23} - \frac{\pi^{\frac{4}{3}}}{4 \cdot 3^{\frac23} \cdot \zeta(3)^{\frac13}} n^{\frac{1}{3}} \right)\left( 1 + \sum_{j=1}^\infty \frac{B_{3,j}}{n^{\frac{j}{3}}}\right) 
	\end{equation*}
	for certain numbers $B_{3,j}$.
	\end{theorem}
\begin{proof}
	We have by \eqref{eq:Lg-Zeta} (see \cite{Apostol}, p. 231)
	\begin{align} \label{eq:L-g3}
		L_{f_3}(s) = \zeta(s)\zeta(s-1) = \sum_{n=1}^\infty \frac{\sigma_1(n)}{n^s}.
	\end{align}
	Thus the only positive poles of $L_{f_3}(s)$ are $\alpha = 2$ and $\beta=1$. Moreover, as $f_3(n) =  \sigma_1(n) \not= 0$ for $n \in \N$, we have $f_3^{-1}(\{0\}) = \emptyset$, and hence we can choose $L$ arbitrarily large. We see that (P1) is satisfied. We have $\mathcal{P} = \{1,2\}$ and these poles are simple, so (P2) is satisfied. Note that we may choose $R$ arbitrarily large, as $\zeta(s)\zeta(s-1)$ has a meromorphic continuation to the entire complex plane.  By properties of the Riemann zeta function, $L_{f_3}(\sigma+it) \ll e^{a|t|}$ on vertical strips with finite width for arbitrary $a > 0$, so in particular (P3) is satisfied. Thus we can use \Cref{paper:BBBF:T:TwoPoleAsymptotics}.  \par
	First, we note that $\lambda=2$ satisfies $\frac{\lambda+1}\lambda \beta < \alpha \leq \frac\lambda{\lambda-1} \beta$. Using \eqref{eq:L-g3}, we compute 
	\begin{align*}
		\omega_{2} = \frac{\pi^{2}}{6}, \quad \omega_{1} = -\frac{1}{2}, \quad L_{f_3}(0) = \frac{1}{24}, \quad 
		L'_{f_3}(0) = \frac{\log(2\pi)}{24}  - \frac{\zeta'(-1)}2.
	\end{align*}
	Thus
	\begin{align} \label{eq:N3CbA}
		C = \frac{e^{L'_{g_3}(0)}\zeta(3)^{\frac{11}{72}}}{\sqrt{2} \cdot 3^{\frac{47}{72}} \cdot \pi^{\frac{7}{36}}}, \quad 
		b = \frac{47}{72}, \quad A_1 = \frac{(3\pi)^{\frac23} \zeta(3)^{\frac13}}{2},  \quad A_2 = - \frac{\pi^{\frac{4}{3}}}{4 \cdot 3^{\frac23} \cdot \zeta(3)^{\frac13}}.
	\end{align}
	Using that $c_{2} = -\frac{\pi^{2}}{12}$ and $c_{1} = \frac{\pi^{2}\zeta(3)}{3}$ we next compute
		\begin{align*}
	A_3 & = -\frac{\pi ^2}{288 \zeta (3)}.
	\end{align*}
	Finally, we find with \Cref{paper:BBBF:T:TwoPoleAsymptotics} and \eqref{eq:N3CbA}
	\begin{align*}
		N_3(n) &\sim \frac{e^{-\frac{\zeta'(-1)}2  -\frac{\pi ^2}{288 \zeta (3)}}\zeta(3)^{\frac{11}{72}}}{2^{\frac{11}{24}} \cdot 3^{\frac{47}{72}} \cdot \pi^{\frac{11}{72}} \cdot n^{\frac{47}{72}}}  \exp\left(  \frac{(3\pi)^{\frac23} \zeta(3)^{\frac13}}{2} n^{\frac23} - \frac{\pi^{\frac{4}{3}}}{4 \cdot 3^{\frac23} \cdot \zeta(3)^{\frac13}} n^{\frac{1}{3}} \right).
	\end{align*}
	With Lemma \ref{lem:Set-L} and \Cref{paper:BBBF:T:main2} we conclude that the exponents in the polynomial terms in the expansions of $N_3(n)$ are given by $\frac{1}{3} \N_0$, as $(\mathcal{M} + \mathcal{N}) \cap [0,\infty) = \frac{1}{3} \N_0$.
\end{proof}

\subsection{The cases $\ell \in \{4,5\}$}\ We are  now ready to determine the asymptotic behavior of $N_4(n)$ and $N_5(n)$.
\begin{theorem} \label{thm:N345} We have, as $n \to \infty$, 
	\begin{align*}
	N_{4}(n) & \sim \frac{e^\frac{\zeta'(-2)}{24} \pi^{\frac{1}{4}} \zeta(3)^{\frac{1}{8}}}{2^{\frac{13}{8}}\cdot 3^{\frac{1}{4}}\cdot 5^{\frac{1}{8}} \cdot n^{\frac{5}{8}}} \exp\left( \frac{2^{\frac{7}{4}} \cdot \pi ^{\frac{3}{2}} \cdot \zeta(3)^{\frac{1}{4}}}{3^{\frac{3}{2}} \cdot 5^{\frac14}}n^{\frac{3}{4}} + A_{4,2}n^{\frac12} + A_{4,3}n^{\frac{1}{4}} + A_{4,4}\right) \left( 1 + \sum_{j=1}^\infty \frac{B_{4,j}}{n^{\frac{j}{4}}}\right),\\
	N_{5}(n) & \sim \frac{e^{\frac{\zeta'(-2)}{2880}}(\pi \zeta(3) \zeta(5))^{\frac{1}{10}}}{2^{\frac25} \cdot 3^{\frac15} \cdot 5^{\frac35} \cdot n^{\frac{3}{5}}} \\
	& \times  \exp\left( \frac{5^{\frac{4}{5}} \cdot \pi ^{\frac{6}{5}} \cdot (\zeta (3) \zeta (5))^{\frac{1}{5}}}{2^{\frac{9}{5}} \cdot 3^{\frac{2}{5}}}n^{\frac{4}{5}} + A_{5,2}n^{\frac{3}{5}} + A_{5,3}n^{\frac{2}{5}} + A_{5,4}n^{\frac{1}{5}} + A_{5,5}\right) \left( 1 + \sum_{j=1}^\infty \frac{B_{5,j}}{n^{\frac{j}{5}}}\right),
	\end{align*}
	with computable constants $A_{4,j}$ ($2 \leq j \leq 4$) and $A_{5,j}$ ($2 \leq j \leq 5$) and certain $B_{4,j}$ and $B_{5,j}$. 
\end{theorem}
\begin{proof} 
	By \Cref{paper:BBBF:T:main2} we have, for $\ell \in \{4,5\}$, as $\a = \ell-1$, $L_{f_\ell}(0) = 0$, and $\omega_{\ell-1} = \z(2) \z(3) \cdots \z(\ell-1)$ (all by \Cref{prop:Poles-of-Lg}) 
	\begin{align*}
	C = \frac{e^{L_{f_\ell}'(0)} (\ell-1)!^{\frac{1}{2\ell}}\sqrt{Z_\ell}}{\sqrt{2\pi\ell}}, \quad b = \frac{\ell+1}{2\ell}.
	\end{align*}
	We find 
	\begin{align*}
	L'_{f_4}(0) = \frac{\z'(-2)}{24}, \quad L'_{g_5}(0) = \frac{\z'(-2)}{2880}.
	\end{align*}
	The constants $A_{4,1}$ and $A_{5,1}$ can now be computed by straightforward calculations. With Lemma \ref{lem:Set-L} and \Cref{paper:BBBF:T:main2} we conclude that the exponents in the polynomial terms in the expansions of $N_\ell(n)$ are given by $\frac{1}{\ell} \N_0$ (for $\ell \in \{4,5\}$), as $(\mathcal{M} + \mathcal{N}) \cap [0,\infty) = \frac{1}{\ell} \N_0$.
\end{proof}

\subsection{Proof of \Cref{thm:Main-Nl}}\hspace{0cm} To prove \Cref{thm:Main-Nl}, we do some preliminary considerations. By Lemma 3.2 of \cite{BBBF23} we have that, as $z \to 0$ in a cone in the right half plane,
	\begin{align} \label{eq:Phi}
	\Phi_{f_\ell}(z) = \sum_{\substack{\nu\in -\mathcal{P}_R\setminus\{0\}}} \mathrm{Res}_{s=-\nu} L_{f_\ell}^*(s)z^{\nu} - L_{f_\ell}(0)\operatorname{Log}(z) + L_{f_\ell}'(0) + O_{R}\left(|z|^R\right)
	\end{align}
and for $k\in\N$
\begin{align} \label{eq:Phi-k}
	\Phi_{f_\ell}^{(k)}(z) = \sum_{\substack{\nu\in -\mathcal{P}_R\setminus\{0\}}} (\nu)_k \mathrm{Res}_{s=-\nu} L_{f_\ell}^*(s)z^{\nu - k} + \frac{(-1)^k(k-1)! L_{f_\ell}(0)}{z^k} + O_{R,k}\left(|z|^{R-k}\right),
\end{align}
where $(s)_k := s(s-1) \cdots (s-k+1)$. By \Cref{prop:Poles-of-Lg}, $L_{f_\ell}$ has exactly three positive poles $\{\ell-3, \ell-2, \ell-1\}$. As $L_{f_\ell}(0) = 0$ by \Cref{prop:Poles-of-Lg} the above becomes
\begin{align} \label{eq:Phi-prime}
	- \Phi'_{f_\ell}(z) = \frac{C_1}{z^\ell} + \frac{C_2}{z^{\ell-1}} + \frac{C_3}{z^{\ell-2}} + O_{R}\left(|z|^{R-1}\right),
\end{align}
where we have, according to \Cref{prop:Poles-of-Lg} and $k=1$ in \eqref{eq:Phi-k},
\begin{align}
	\nonumber C_1 &= (\ell-1)! \zeta(\ell)\prod_{\substack{0\le k\le\ell-2 \\ k \not= \ell-2}} \zeta\left(\ell-1 - k\right) > 0, \quad 
	C_2  = (\ell-2)! \zeta(\ell-1)\prod_{\substack{0\le k\le\ell-2 \\ k \not= \ell-3}} \zeta(\ell-2 - k) \ne 0,\\
	\label{eq:C123} C_3 &=  (\ell-3)! \zeta(\ell-2)\prod_{\substack{0\le k\le\ell-2 \\ k \not= \ell-4}} \zeta(\ell-3- k) \ne 0.
\end{align}
\par 
To calculate the exponent sets $\mathcal{L}, \mathcal{M},$ and $\mathcal{N}$, we apply \Cref{prop:BBBF:paper:cor:rhosaddleexp1}. Let $\alpha := \ell-1$. By \Cref{prop:BBBF:paper:cor:rhosaddleexp1} the exponents of the saddle point function $\varrho_{\ell,n}$ belonging to $f_\ell$ lie in $\frac{1}{\ell} \N$. Our next goal is to prove the following version of Lemma 4.3 of \cite{BBBF23}.
\begin{lemma} \label{lem:rho-expand} For $\ell\ge4$, let $f := f_{\ell}$. Then, as $n\to\infty$,
	\begin{equation*}
		\p_{\ell,n} \sim \sum_{j=1}^{\infty} \frac{K_{\ell,j}}{n^{\frac{j}{\ell}}} 
	\end{equation*}
	for some constants $K_{\ell,j}$ independent of $n$ that can be calculated explicitly using $a_k(x)$, $b_k(x)$, and $e_h$ defind below in \eqref{eq:ak(x)}, \eqref{eq:bk(x)}, and \eqref{eq:w(x)}, respectively. We have 
	\begin{align} \label{eq:K1ell}
	K_{\ell,1} = C_1^{\frac{1}{\ell}}.
	\end{align}
\end{lemma}
\begin{proof} Note that it suffices to treat all of the following asymptotic expansions formally, as in Section 3 of \cite{BBBF23} it was shown that the analytical considerations can be made rigorous under the assumptions of \Cref{paper:BBBF:T:main2}.\par
	By \Cref{prop:BBBF:paper:cor:rhosaddleexp1}, the exponents of $\varrho_{\ell,n}$ lie in $\frac{1}{\ell}\N$. Hence, by Proposition 3.3 of \cite{BBBF23} there exist $K_{\ell,j} \in \R$, such that
	\begin{align*}
		\varrho_{\ell,n} \sim \sum_{j=1}^\infty \frac{K_{\ell,j}}{n^{\frac{j}{\ell}}} \qquad (n \to \infty).
	\end{align*}
	With \eqref{eq:Phi-prime} we obtain $a_{-\Phi'_{f_\ell},1} = C_1$. We also have $a_{\varrho_\ell,1} = a_{-\Phi'_{f_\ell},1}^{\frac{1}{\ell}}$ by \Cref{prop:BBBF:paper:cor:rhosaddleexp1}, since $\alpha +1 = \ell$. We conclude that $K_{\ell,1} = a_{\varrho_\ell,1} = C_1^{\frac{1}{\ell}}$. 
		
	 For the exact values of the coefficients $K_{\ell,j}$ we need to resolve 
	\begin{align} \label{eq:C1C2C3-rho}
		\frac{C_1}{\varrho_{\ell,n}^{\ell}} + \frac{C_2}{\varrho_{\ell,n}^{\ell-1}} + \frac{C_3}{\varrho_{\ell,n}^{\ell-2}} - n = 0
	\end{align}
	by \eqref{eq:Phi-prime} and \Cref{prop:BBBF:paper:cor:rhosaddleexp1}. In order to do so, we divide \eqref{eq:C1C2C3-rho} by $n$ to obtain
	\begin{align*}
		C_1 \left( \frac{n^{-\frac{1}{\ell}}}{\varrho_{\ell,n}}\right)^\ell + C_2 n^{-\frac{1}{\ell}} \left( \frac{n^{-\frac{1}{\ell}}}{\varrho_{\ell,n}}\right)^{\ell-1} + C_3 n^{-\frac{2}{\ell}} \left( \frac{n^{-\frac{1}{\ell}}}{\varrho_{\ell,n}}\right)^{\ell-2}   - 1 = 0.
	\end{align*}
	Thus we are interested in points $(x,z) = (x,z(x))$ on the curve
	\begin{align*}
		C_1 z^\ell + C_2 x z^{\ell-1} + C_3 x^2  z^{\ell-2} -1 = 0 
	\end{align*}
	for small values of $x$, where $z = \frac{n^{-\frac{1}{\ell}}}{\varrho_{\ell,n}}$. Making the change of variables 
	\begin{align} \label{eq:Change}
	z =: w + C_1^{-\frac{1}{\ell}} 
	\end{align}
	we obtain 
	\begin{equation*}
		C_1 \sum_{k=0}^{\ell} {\ell \choose k} w^{k} C_1^{-\frac{\ell-k}{\ell}} + C_2 x \sum_{k=0}^{\ell-1} {\ell-1 \choose k} w^{k} C_1^{-\frac{\ell - 1-k}{\ell}} + C_3 x^2 \sum_{k=0}^{\ell-2} {\ell-2 \choose k} w^{k} C_1^{-\frac{\ell-2-k}{\ell}} - 1= 0,
	\end{equation*}
	using the Binomial Theorem. We rewrite this in the shape:
	\begin{align*} 
		\sum_{k=0}^{\ell} a_k(x) w^k = 0,
	\end{align*}
	where
	\begin{align} \label{eq:ak(x)}
		a_k(x) :=  {\ell \choose k} C_1^{\frac{k}{\ell}}  + C_2  {\ell-1 \choose k} C_1^{\frac{k+1}{\ell}-1}x  + C_3  {\ell-2 \choose k} C_1^{\frac{k+2}{\ell}-1}x^2  - \delta_{k=0},
	\end{align}
	where $\delta_{\mathcal S}=1$ if a statement $\mathcal S$ holds and $\delta_{\mathcal S} = 0$ otherwise. In other words, each $a_k(x)$ is a polynomial in $x$ of degree at most $2$. As 
		$$a_0(x) =  C_2C_1^{\frac{1}{\ell}-1}x + C_3C_1^{\frac{2}{\ell}-1}x^2,$$ 
	we obtain the identity
	\begin{align} \label{eq:Power-Zero}
		\varphi_x(w) := \sum_{k=1}^{\ell} a_k(x) w^k = -a_0(x) = - C_2C_1^{\frac{1}{\ell}-1}x - C_3C_1^{\frac{2}{\ell}-1}x^2.
	\end{align}
	We now apply \Cref{lem:power-series-inverse} to \eqref{eq:Power-Zero} to Taylor approximate its root close to $w=0$ as $x \to 0$. As $C_1 \not= 0$ (we even show $C_{1}>0$ in \eqref{eq:C123}), we have, by definition
	\begin{equation}\label{eq:a1x}
		a_1(x) = \ell C_1^{\frac{1}{\ell}} + (\ell-1)C_2C_1^{\frac{2}{\ell}-1}x + (\ell-2)C_3C_1^{\frac{3}{\ell}-1}x^2\not= 0.
	\end{equation}
	for all $x$ sufficiently small (i.e., all $n$ sufficiently large). Thus $\varphi_x(0) = 0$ and $\varphi_x'(0) = a_1(x) \not= 0$. Thus we may apply \Cref{lem:power-series-inverse} and obtain that the coefficients of the inverse function are 
	\begin{align}\label{eq:bk(x)}
	b_k(x) &= \frac{1}{ka_1(x)^k} \sum_{\substack{\ell_1,\ell_2, \ell_3 \dots \geq 0 \\ \ell_1 + 2\ell_2 + \dots = k-1}} (-1)^{\ell_1 + \ell_2 +\ell_3+\dots} \frac{k \cdots (k-1+\ell_1+\ell_2+\dots)}{\ell_1! \ell_2! \ell_3! \cdots}\\
	\nonumber & \hspace{5cm} \times \left( \frac{a_2(x)}{a_1(x)}\right)^{\ell_1} \left( \frac{a_3(x)}{a_1(x)}\right)^{\ell_2} \cdots.
	\end{align}
As a result, for $x$ sufficiently small, we can solve \eqref{eq:Power-Zero} with
\begin{align} \label{eq:w(x)}
w(x) = \sum_{k = 1}^\infty b_k(x) \left( - C_2C_1^{\frac{1}{\ell}-1}x - C_3C_1^{\frac{2}{\ell}-1}x^2\right)^k =: \sum_{h=0}^\infty e_h x^h
\end{align}
with $e_h \in \R$.  We now have an expansion for $w = z - C_1^{-\frac{1}{\ell}}$ in powers of $n^{-\frac{1}{\ell}}$. Resolving \eqref{eq:Change} (recall $z = \frac{n^{-\frac{1}{\ell}}}{\varrho_{\ell,n}}$), i.e., 
	\begin{align*}
		\frac{n^{-\frac{1}{\ell}}}{\varrho_{\ell,n}} = C_1^{-\frac{1}{\ell}}  + w\left(n^{-\frac{1}{\ell}}\right)
	\end{align*}
	leads to the closed expression 
	\begin{equation*}
		\varrho_{\ell,n} = \frac{n^{-\frac{1}{\ell}}}{C_1^{-\frac{1}{\ell}}  + w\left(n^{-\frac{1}{\ell}}\right)} = C_1^{\frac{1}{\ell}} n^{-\frac{1}{\ell}} \sum_{m=0}^\infty C_1^{\frac{m}{\ell}} w\left(n^{-\frac{1}{\ell}}\right)^m.
	\end{equation*}
	The lemma now follows as the previous equation gives the desired expansion.
\end{proof}
We are now ready to prove \Cref{thm:Main-Nl}.
\begin{proof}[Proof of \Cref{thm:Main-Nl}]
	We identify the part of the asymptotic in $$e^{n\varrho_{\ell,n}}G_{f_\ell}(\varrho_{\ell,n}) = e^{n\varrho_{\ell,n} + \Phi_{f_\ell}(\varrho_{\ell,n})}$$ with non-negative exponents. In other words, we ask for the term
	\begin{equation} \label{eq:Main-Asy}
		e^{n\varrho_{\ell,n} + \Phi_{f_\ell}(\varrho_{\ell,n})} = \exp\left( \left[ n\varrho_{\ell,n} + \Phi_{f_\ell}(\varrho_{\ell,n})\right]_*\right)(1 + o(1)),
	\end{equation}	
	where the notation $[]_*$ means the part of an asymptotic expansion $\sum_{j=1}^\infty a_j n^{\beta_j}$ with $\beta_j \geq 0$. 
	First we get with \Cref{lem:rho-expand}, as $n \to \infty$,
	\begin{align} \label{eq:nrhon}
	n\varrho_{\ell,n} = \sum_{j=1}^{\ell} K_{\ell,j}n^{\frac{\ell-j}{\ell}} + O\left(n^{-\frac{1}{\ell}}\right).
	\end{align}
		Moreover, with \eqref{eq:Phi} and integrating \eqref{eq:Phi-prime} we get as $z \to 0^+$ (in a cone) 
	\begin{equation*}
	\Phi_{f_\ell}(z) = \frac{C_1}{(\ell-1)z^{\ell-1}} + \frac{C_2}{(\ell-2)z^{\ell-2}} + \frac{C_3}{(\ell-3)z^{\ell-3}} + L'_{g_\ell}(0) + o(1),
	\end{equation*}
	since $L_{f_{\ell}}(0)=0$.
	Thus in particular, as $n \to \infty$,
	\begin{equation} \label{eq:Phi-Expansion}
	\Phi_{f_\ell}(\varrho_{\ell,n}) = \frac{C_1}{(\ell-1) \varrho_{\ell,n}^{\ell-1}} + \frac{C_2}{(\ell-2) \varrho_{\ell,n}^{\ell-2}} + \frac{C_3}{(\ell-3)\varrho_{\ell,n}^{\ell-3}} + L'_{g_\ell}(0) + o(1).
	\end{equation}
	We next use \eqref{eq:K1ell} and rewrite 
	\begin{align*}
		\frac{1}{\varrho_{\ell,n}} & \sim \frac{1}{C_1^{\frac{1}{\ell}}n^{-\frac{1}{\ell}} + \sum_{j=2}^\infty K_{\ell,j}n^{-\frac{j}{\ell}}} = \frac{n^{\frac{1}{\ell}}}{C_1^{\frac{1}{\ell}}}  \left( 1 + \frac{1}{C_1^{\frac{1}{\ell}}} \sum_{j=1}^\infty \frac{K_{\ell,j+1}}{n^{\frac{j}{\ell}}} \right)^{-1}.
	\end{align*}
	We can work with the following identity regarding formal power series: 
	\begin{equation*}
		\frac{1}{1 + \sum_{k=1}^\infty b_k z^k} = \sum_{m=0}^\infty (-1)^m \left( \sum_{k=1}^\infty b_k z^k\right)^m = 1 + \sum_{n=1}^\infty c_n z^n,
	\end{equation*}
	with
	\begin{equation*}
		c_n = \sum_{\substack{j_1, j_2, \ldots, j_{n} \geq 0  \\ j_1 + 2j_2 + 3j_3 + \dots = n}} (-1)^{j_1+j_2+\dots} {j_1 + j_2 + j_3 + \dots \choose {j_1, j_2, j_3, \ldots }}  b_1^{j_1} b_2^{j_2} b_3^{j_3} \cdots.
	\end{equation*}
	We thus obtain
	\begin{equation*}
		\frac{1}{\varrho_{\ell,n}} \sim n^{\frac{1}{\ell}} \sum_{m=0}^\infty \frac{D_{m}}{n^{\frac{m}{\ell}}},
	\end{equation*}
	where 
	\begin{equation*} 
		D_{m} :=  C_1^{-\frac{1}{\ell}}\sum_{\substack{j_1, j_2, \ldots \geq 0  \\ j_1 + 2j_2 + 3j_3 + \dots = m}} (-1)^{j_1+j_2+\dots} {j_1 + j_2 + j_3 + \dots \choose {j_1, j_2, j_3, \ldots }}  C_1^{-\frac{m}{\ell}} K_{\ell,2}^{j_1} K_{\ell,3}^{j_2} K_{\ell,4}^{j_3} \cdots.
	\end{equation*}
	Using the Multinomial Theorem again, we find
	\begin{align}
		\nonumber & \hspace{-2cm} \frac{C_1}{(\ell-1) \varrho_{\ell,n}^{\ell-1}} + \frac{C_2}{(\ell-2) \varrho_{\ell,n}^{\ell-2}} + \frac{C_3}{(\ell-3)\varrho_{\ell,n}^{\ell-3}} \\
		\nonumber  & \hspace{-1cm}=\frac{C_1}{\ell-1} n^{\frac{\ell-1}{\ell}} \sum_{k=0}^\infty n^{-\frac{k}{\ell}} \sum_{\substack{m_1, m_2, \ldots \geq 0 \\ m_1 + m_2 + \dots + \dots = \ell-1 \\ m_2 + 2m_3 + \dots = k}} {\ell - 1 \choose m_1, m_2, \ldots} D_{0}^{m_1} D_{1}^{m_2} \cdots  \\
		\label{eq:Phi-expanding} &\quad \hspace{-1cm}+\frac{C_2}{\ell-2} n^{\frac{\ell-2}{\ell}} \sum_{k=0}^\infty n^{-\frac{k}{\ell}} \sum_{\substack{m_1, m_2, \ldots \geq 0\\ m_1 + m_2 + \dots + \dots = \ell-2 \\ m_2 + 2m_3 + \dots = k}} {\ell - 2 \choose m_1, m_2, \ldots} D_{0}^{m_1} D_{1}^{m_2} \cdots  \\
		\nonumber  &\quad \hspace{-1cm}+\frac{C_3}{\ell-3} n^{\frac{\ell-3}{\ell}} \sum_{k=0}^\infty n^{-\frac{k}{\ell}} \sum_{\substack{m_1, m_2, \ldots \geq 0\\ m_1 + m_2 + \dots + \dots = \ell-3 \\ m_2 + 2m_3 + \dots = k}} {\ell - 3 \choose m_1, m_2, \ldots} D_{0}^{m_1} D_{1}^{m_2} \cdots  .
		\end{align} \par
	We next determine closed formulas for the exponent coefficients $A_{\ell,k}$. We first find with \Cref{paper:BBBF:T:main2} 
	\begin{align} \label{eq:A1}
		A_{\ell,1} & = \left( 1 + \frac{1}{\ell-1}\right) (\zeta(2) \cdots \zeta(\ell-1) (\ell-1)! \zeta(\ell))^{\frac{1}{\ell}}.
	\end{align}
		Also note that, combing \Cref{paper:BBBF:T:main2} and \Cref{lem:BBBF:paper:L:mainasy}, we obtain
	\begin{multline*}
		\frac{e^{n\p_{\ell,n}}G_{f_\ell}(\p_{\ell,n})}{\sqrt{2\pi}}\left(\sum\limits_{j=1}^N \frac{d_j }{n^{\nu_j}} + O_{L,R}\left(n^{-\min\left\{\frac{L+1}{\a+1}, \frac{R+\a}{\a+1}+\frac{\a+2}{2(\a+1)}\right\}}\right)\right) \\
		= \frac{C}{n^b}
		\exp\left(A_1n^\frac{\a}{\a+1}+\sum_{j=2}^M A_jn^{\a_j}\right)\left(1+\sum\limits_{j=2}^N \frac{B_j}{n^{\b_j}} + O_{L,R}\left(n^{-\min\left\{\frac{2L-\a}{2(\a+1)},\frac{R}{\a+1}\right\}}\right)\right),
	\end{multline*}
		for some numbers $0 \leq \alpha_M < \cdots < \alpha_2 < \frac{\alpha}{\alpha+1}$. Consequently with \eqref{eq:Main-Asy}, comparing the exponential terms with powers in $\mathcal{L} \cap [0,\infty)$ (note that the term $L_{f_\ell}'(0)$ is treated separately, so we have to subtract it),
	\begin{align}
		\label{eq:First}  \left[n\p_{\ell,n} + \Phi_{f_\ell}(\varrho_{\ell,n}) - L'_{g_\ell}(0)\right]_* & = A_1n^\frac{\a}{\a+1}+\sum_{j=2}^M A_jn^{\a_j} =:  \sum_{k=1}^\ell A_{\ell,k} n^{\frac{\ell-k}{\ell}}.
	\end{align}
	Using \Cref{lem:Set-L} employing \eqref{eq:nrhon}, \eqref{eq:Phi-Expansion}, and \eqref{eq:Phi-expanding}, to obtain for $k \geq 2$,
	\begin{align} 
		A_{\ell,k}  &= K_{\ell,k} + \frac{C_1}{\ell-1} \hspace{-0.4cm} \sum_{\substack{m_1, m_2, \ldots \geq 0 \\ m_1 + m_2 + \dots + \dots = \ell-1 \\ m_2 + 2m_3 + \dots = k-1}} \hspace{-0.4cm} {\ell - 1 \choose m_1, m_2, \ldots} D_{0}^{m_1} D_{1}^{m_2} \cdots\nonumber\\
		&\hspace{5cm}+ \frac{C_2}{\ell-2} \hspace{-0.4cm} \sum_{\substack{m_1, m_2, \ldots \geq 0\\ m_1 + m_2 + \dots + \dots = \ell-2 \\ m_2 + 2m_3 + \dots = k-2}} \hspace{-0.4cm} {\ell - 2 \choose m_1, m_2, \ldots} D_{0}^{m_1} D_{1}^{m_2} \cdots\nonumber\\
		&\hspace{5cm}+ \frac{C_3}{\ell-3} \hspace{-0.4cm} \sum_{\substack{m_1, m_2, \ldots \geq 0\\ m_1 + m_2 + \dots + \dots = \ell-3 \\ m_2 + 2m_3 + \dots = k-3}} \hspace{-0.4cm} {\ell - 3 \choose m_1, m_2, \ldots} D_{0}^{m_1} D_{1}^{m_2}\cdots\hspace{-0.05cm}.\label{eq:Ak} 
	\end{align}
	If $\ell \geq 6$, then we have $L'_{f_\ell}(0) = 0$, since $L_{f_\ell}$ has a zero in $s=0$ of order at least 2. This simplifies the formula to the one in Theorem \ref{thm:Main-Nl}. \Cref{paper:BBBF:T:main2}, \Cref{prop:Poles-of-Lg}, and  
	\begin{equation*}
		C  = \frac{(\ell-1)!^{\frac{1}{2\ell}} Z_\ell^{\frac12}}{\sqrt{2\pi \ell}} 
	\end{equation*}
	as well as \eqref{eq:A1} gives the value of $A_{\ell,1}$.
		With Lemma \ref{lem:Set-L} and \Cref{paper:BBBF:T:main2} we conclude that the exponents in the polynomial terms in the expansions of $N_\ell(n)$ are given by $\frac{1}{\ell} \N_0$, as $(\mathcal{M} + \mathcal{N}) \cap [0,\infty) = \frac{1}{\ell} \N_0$.
\end{proof}
\section{Log-Concavity and the proof of \Cref{thm:log}}\label{S:Log}
\subsection{Proof of \Cref{thm:log}}\hspace{0cm} In this subsection, we prove our general result on log-concavity.
\begin{proof}[Proof of \Cref{thm:log}]
	We have
	\begin{align}
		\label{eq:log-concave-1} & c(n)^2 - c(n+1) c(n-1) \sim C^2\vast(\frac{\exp\left(2\sum_{\lambda\in\mathcal S}A_\lambda n^\lambda\right)}{n^{2\kappa}}\left(\sum_{\mu\in\mathcal T} \frac{\beta_\mu}{n^\mu}\right)^2\\
		\nonumber & \hspace{2cm} -\frac{\exp\left(\sum_{\lambda\in\mathcal S}A_\lambda\left((n+1)^\lambda+(n-1)^\lambda\right)\right)}{(n+1)^\kappa(n-1)^\kappa}\sum_{\mu\in\mathcal T}\frac{\beta_\mu}{(n+1)^\mu}\sum_{\nu\in\mathcal T}\frac{\beta_\nu}{(n-1)^\nu}\vast).
	\end{align}
	We now claim that for all $\lambda \in \mathcal{S}$, we have 
	\begin{equation} \label{eq:Claim-1}
		\exp\left(A_\lambda\left((n+1)^\lambda+(n-1)^\lambda-2n^\lambda\right)\right) = 1 + \frac{\gamma_{\lambda,1}}{n^{2-\lambda}} + o\left(n^{-2+\lambda}\right)
	\end{equation}
	for $\gamma_{\lambda,1} \in \R$. To see this, write $\lambda = \frac{a}{b}$ with $\gcd(a,b)=1$ and set $x:=n^{-\frac1b}$. Then
	\begin{align*}
		(n+1)^\frac ab+(n-1)^\frac ab-2n^\frac ab &= \left(\frac1{x^b}+1\right)^\frac ab + \left(\frac1{x^b}-1\right)^\frac ab - 2x^{-a}\\
		&= \frac1{x^a} \left(\left(1+x^b\right)^\frac ab+\left(1-x^b\right)^\frac ab-2\right) =: f_{a,b}(x).
	\end{align*}
	Now, as $y \to 0$,
	\begin{equation*}
		g_\lambda(y) := (1+y)^\lambda+(1-y)^\lambda-2 = O\left(y^2\right).
	\end{equation*}
	Thus $f_{a,b}$ has a Taylor expansion of the shape
	\begin{equation*}
		f_{a,b}(x) = \sum_{j\ge1} \alpha_{\lambda,j} x^{2b j-a}.
	\end{equation*}
	Thus
	\begin{equation*}
		f_{a,b}\left(n^{-\frac1b}\right) = \sum_{j\ge1} \frac{\alpha_{\lambda,j}}{n^{\frac{2b j-a}b}} = \frac{\alpha_{\lambda,1}}{n^{2-\lambda}}+o\left(n^{-2+\lambda}\right).
	\end{equation*}
	Plugging in the expansion of the exponential function gives the claim, using that $0<\lambda < 1$. 
	
	Using \eqref{eq:Claim-1}, we obtain
	\begin{align*}
		\exp\left(\sum_{\lambda\in\mathcal S}A_\lambda\left((n+1)^{\lambda}+(n-1)^{\lambda}-2n^\lambda\right)\right) &= \prod_{\lambda \in \mathcal{S}} \left( 1 + \frac{\gamma_{\lambda,1}}{n^{2-\lambda}} + o\left(n^{-2+\lambda}\right)\right) \\
		&= 1 + \frac{\gamma_{\lambda^*,1}}{n^{2-\lambda^*}} + o\left( n^{-2+\lambda^*}\right).
	\end{align*}
	Note that $\gamma_{\lambda^*,1}=A_{\lambda^*}\alpha_{\lambda^*,1}$ and that $\alpha_{\lambda,1} = \lambda(\lambda-1)<0$.
	
	Next we claim that (for $n > 1$)
	\begin{equation} \label{eq:kappa-binomial}
		\frac1{(n+1)^\kappa(n-1)^\kappa} =\frac1{\left(n^2-1\right)^\kappa} = n^{-2\kappa} \left(1+\sum_{j\ge1}\frac{\delta_j}{n^{2j}}\right)
	\end{equation}
	for certain $\delta_j$. To see this, set $x:=n^{-2}$. Then we want
	\begin{equation*}
		\frac{x^{-\kappa}}{\left(\frac1x-1\right)^\kappa} = 1 + \sum_{j\ge1} \delta_j x^j.
	\end{equation*}
	The left-hand side is $\frac1{(1-x)^\kappa}$ and the claim follows. 
	
	Thus, by \eqref{eq:log-concave-1},
	\begin{align*}
		&c(n)^2 - c(n+1) c(n-1)\\
		&\sim C^2 \frac{\exp\left(2\sum_{\lambda\in\mathcal S}A_\lambda n^\lambda\right)}{n^{2\kappa}} \vast(\left(\sum_{\mu\in\mathcal T}\frac{\beta_\mu}{n^\mu}\right)^2\\
		&\hspace{1.3cm}- \left(1+\frac{\gamma_{\lambda^*,1}}{n^{2-\lambda^*}} + o\left(n^{-2+\lambda^*}\right)\right)\left(1+O\left(n^{-2}\right)\right) \sum_{\mu\in\mathcal T} \frac{\beta_\mu}{(n+1)^\mu} \sum_{\nu\in\mathcal T} \frac{\beta_\nu}{(n-1)^\nu}\vast).
	\end{align*}

	The sign of this is dictated by
	\begin{equation*}
		\sum_{\mu,\nu\in\mathcal T} \frac{\beta_\mu\beta_\nu}{n^{\mu+\nu}} - \left(1+\frac{\gamma_{\lambda^*,1}}{n^{2-\lambda^*}} + o\left(n^{-2+\lambda^*}\right)\right)\left(1+O\left(n^{-2}\right)\right) \sum_{\mu\in\mathcal T} \frac{\beta_\mu}{(n+1)^\mu} \sum_{\nu\in\mathcal T} \frac{\beta_\nu}{(n-1)^\nu}.
	\end{equation*}
	As $\beta_0 = 1$, the above equals mod $(o(n^{-2+\lambda^*}))$,
	\begin{align}
		&\sum_{\substack{\mu,\nu\in\mathcal T\\0\le \mu+\nu\le2-\lambda^*}} \frac{\beta_\mu\beta_\nu}{n^{\mu+\nu}} - \sum_{\substack{\mu,\nu\in\mathcal T\\0\le \mu+\nu\le2-\lambda^*}} \frac{\beta_\mu\beta_\nu}{(n+1)^\mu(n-1)^\nu} - \frac{\gamma_{\lambda^*,1}}{n^{2-\lambda^*}}\nonumber\\
		&\hspace{1cm}= 2\sum_{\substack{\mu\in\mathcal T\\0\le \mu\le1-\frac{\lambda^*}2}} \beta_\mu^2 \left(\frac1{n^{2\mu}}-\frac1{(n+1)^\mu(n-1)^\mu}\right)\label{E:uni}\\ 
		&\hspace{2cm}+ \sum_{\substack{\mu\in\mathcal T\\0\le\mu<\nu\\1\le \mu+\nu\le2-\lambda^*}} \beta_\mu \beta_\nu \left(\frac2{n^{\mu+\nu}}-\frac1{(n+1)^\mu(n-1)^\nu}-\frac1{(n+1)^\nu(n+1)^\mu}\right) - \frac{\gamma_{\lambda^*,1}}{n^{2-\lambda^*}}.\nonumber
	\end{align}
	Next, by \eqref{eq:kappa-binomial} we obtain that (for $n > 1$)
	\begin{equation*}
		\frac1{(n+1)^\mu(n-1)^\mu} 
		= n^{-2\mu} + O\left(n^{-2\mu-2}\right).
	\end{equation*}
	Similarly, we see that, for certain $\rho_r$,
	\begin{equation*}
		\frac1{(n+1)^\mu(n-1)^\nu} + \frac1{(n+1)^\nu(n-1)^\mu} = 2n^{-\mu-\nu} \left(1+\sum_{r\ge1}\frac{\rho_r}{n^{2r}}\right).
	\end{equation*}
	Thus \eqref{E:uni} becomes
	\begin{equation*}
		O\left(n^{-2}\right) - \frac{\gamma_{\lambda^*,1}}{n^{2-\lambda^*}}.
	\end{equation*}
	To conclude the claim, we note that we have
	\begin{equation*}
		\sgn(-\gamma_{\lambda^*,1}) = \sgn(A_{\lambda^*}) = 1.\qedhere
	\end{equation*}
\end{proof}
\subsection{Examples} \label{subsec:Examples}
We can show the following result.
\begin{corollary} \label{cor:pf} 
	Let $f \colon \N \to \N_0$ satisfy all conditions of \Cref{paper:BBBF:T:main2}, and assume that we can choose $L$ in \ref{paper:main:1} arbitrarly large. Assume furthermore that $L_f(s)$ has a meromorphic continuation to $\C$ with only rational poles. Then $p_f(n)$ is log-concave for $n$ sufficiently large. 
\end{corollary}
\begin{proof} 
	\Cref{paper:BBBF:T:main2} provides an expansion for $p_f(n)$ as in \Cref{thm:log}, since the rational poles of $L_f(s)$ guarantee that all exponents occuring in the expansion of $p_f(n)$ are again rational. Note that $\lambda^* = \frac{\alpha}{\alpha+1}$ and\footnote{Note the abuse of notation.} $A_1 > 0$ again by \Cref{paper:BBBF:T:main2}. This gives then the claim.
\end{proof}
There are several further applications. Examples are, besides the classical partitions and plane partitions, partitions into $k$-gonal numbers $p_k(n)$, and the number of $n$-dimensional representations for the groups $\mathfrak{su}(3)$ and $\mathfrak{so}(5)$, denoted by $r_{\mathfrak{su}(3)}(n)$ and $r_{\mathfrak{so}(5)}(n)$, respectively. The asymptotic behavior of the numbers $r_{\mathfrak{su}(3)}(n)$ was first studied by Romik \cite{Ro} and later refined by two of the authors in \cite{BF}. Asymptotic expressions for $r_{\mathfrak{so}(5)}(n)$ and $p_k(n)$ were given in \cite{BBBF23}. We directly obtain the following:
\begin{corollary} \label{cor:pkrr} 
	Let $k \geq 3$. For $n$ sufficiently large, the sequences $p_k(n)$, $r_{\mathfrak{su}(3)}(n)$, and $r_{\mathfrak{so}(5)}(n)$ are log-concave. 
\end{corollary}
Finally, we give another example in a slightly different direction.
 For $d \in \N_0$, let\footnote{see \cite{HN22}, where a generalization of generating function for the partition and plane partition function was studied}
\begin{align}\label{def:nd}
	\sum_{n \geq 0} \mathrm{p}_d(n)q^n := \prod_{n \geq 1} \left( 1 - q^n\right)^{-n^d}.
\end{align}
Note that we have $\mathrm{p}_0(n) = p(n)$ and $\mathrm{p}_1(n) = \mathrm{pp}(n)$ is the plane partition function. Note that the corresponding $L$-series is given by 
\begin{align*}
	\sum_{n \geq 1} \frac{n^d}{n^s} = \zeta(s-d).
\end{align*}
It is not hard to check that the conditions \ref{paper:main:1}, \ref{paper:main:3}, and \ref{paper:main:4} are satisfied, and we have 
\begin{align*}
	\mathcal{P}_R \subseteq \{d+1\} \cup \{-1,-2,-3, \ldots\}
\end{align*} 
for the set of poles $s \not= 0$ of $L_d^*(s) := \Gamma(s)\zeta(s+1)\zeta(s-d)$. Note that $\alpha = d+1$ and
\begin{align*}
	b & = \frac{1}{2(d+2)} - \frac{\zeta(-d)}{d+2} + \frac12, \quad A_1 = \left( 1 + \frac{1}{d+1}\right)\left( (d+1)! \zeta(d+2) \right)^{\frac{1}{d+2}}, \\
	C & = \frac{e^{\zeta'(-d)}((d+1)!\z(d+2))^\frac{\frac12-\z(-d)}{d+2}}{\sqrt{2\pi(d+2)}}.
\end{align*}
As a consequence, we have by Theorem \ref{paper:BBBF:T:main2}
\begin{align*}
	\mathrm{p}_d(n) \sim \frac{C}{n^b} e^{A_1n^{\frac{d+1}{d+2}}} \left( 1 + \sum_{j \geq 1} \frac{E_{d,j}}{n^{\frac{j}{d+2}}}\right)
\end{align*}
with certain $E_{d,j}$. Note that, by the same arguments used for the other examples above, $\mathrm{p}_d(n)$ is log-concave for $n$ sufficiently large. In \cite{HNT23} it had been proven that in the case of
plane partitions, $p_1(n)$ is log-concave for almost all $n$ and
conjectured that this is already valid for $n \geq 12$.
This conjecture had been proven by Ono, Pujahari, and Rolen
\cite{OPR22}.

\section{Proof of \Cref{thm:BO}} \label{sect:BO}
The idea of the following proof is similar to that of \Cref{thm:log}.
\begin{proof}[Proof of \Cref{thm:BO}]
	Assume that $a$, $b\gg1$. Then 
	\begin{multline*}
	c(a) c(b) - c(a+b) \sim \frac{C^2}{a^\kappa b^\kappa} \exp\left(\sum_{\lambda\in\mathcal S} A_\lambda \left(a^\lambda+b^\lambda\right)\right) - \frac C{(a+b)^\kappa} \exp\left(\sum_{\lambda\in\mathcal S} A_\lambda (a+b)^\lambda\right)\\
	= \frac{C^2}{a^\kappa b^\kappa} \exp\left(\sum_{\lambda\in\mathcal S} A_\lambda \left(a^\lambda+b^\lambda\right)\right) \left(1-\frac{a^\kappa b^\kappa}{C(a+b)^\kappa}\exp\left(\sum_{\lambda\in\mathcal S} A_\lambda \left((a+b)^\lambda-a^\lambda-b^\lambda\right)\right)\right).
	\end{multline*}
	Without loss of generality we may assume that $a\ge b$ and write $a=xb$, $x\ge1$. Then 
	\begin{multline*}
	\frac{a^\kappa b^\kappa}{C(a+b)^\kappa}\exp\left(\sum_{\lambda\in\mathcal S} A_\lambda \left((a+b)^\lambda-a^\lambda-b^\lambda\right)\right)\\
	= \frac{x^\kappa b^\kappa}{C(x+1)^\kappa}\exp\left(\sum_{\lambda\in\mathcal S} A_\lambda b^\lambda \left((1+x)^\lambda-x^\lambda-1\right)\right).
	\end{multline*}
	Let $f_\lambda(x):=(1+x)^\lambda-x^\lambda-1$. It is not hard to see that $f'_\lambda(x) < 0$. As $f_\lambda(0)=0$, we have $f_\lambda(x)<0$ for $x\ge1$. Moreover , using l'Hospital, 
	\begin{align*}
	\lim_{x\to\infty} \left((x+1)^\lambda-x^\lambda-1\right) = -1.
	\end{align*}
	We conclude $-1 \leq f_\lambda(x) < 0$ for all $x \geq 1$.
	This gives the claim.
\end{proof}

\section{Open questions}
We leave some open questions to the interested reader.
\begin{enumerate}[leftmargin=19pt, labelwidth=!, labelindent=0pt]
	\item What can be said about Turán type inequalities for the sequences studied? Note that the case $\ell=2$ of the partition function was treated in (\cite{GORZ19}, Theorems 1 and 2) by showing that certain classes of Jensen polynomials have real roots. The more general case $\ell\ge3$ will be much harder.
	\item It would be interesting to make the results of this paper (in particular log-concavity) explicit. The hard part would be to make \cite{BBBF23} explicit. For example, the case $\ell=2$ of the partition function was treated in (\cite{DP15}, Theorem 1.1). Note that this turns out to be much easier than the higher cases of $\ell$ as in the partition case one has an exact formula, which yields inequalities for the partition function.
	\item According to a conjecture by Chen (see Conjecture 1.2 in \cite{DP15}), the sequence $|C_{2,n}| = n! p(n)$ is log-convex for all $n > 1$, which was shown by DeSalvo and Pak 2015 (\cite{DP15}, Theorem 4.1). The question now is whether this is also true for sufficiently large $n$ for the $|C_{\ell,n}|$ for all $\ell \geq 3$, and to what extent this could be related to the results of this work.
\end{enumerate}

\end{document}